\documentclass[12pt]{amsart}
\usepackage[utf8]{inputenc}
\usepackage{amssymb}
\usepackage{amsmath}
\usepackage{amsfonts}
\usepackage{amsthm}
\usepackage{tikz}
\usepackage{algorithmic}
\usepackage{asymptote}
\usepackage{url}
\newtheorem{theorem}{Theorem}[section]
\newtheorem{corollary}{Corollary}[theorem]
\newtheorem{lemma}[theorem]{Lemma}
\newtheorem{proposition}[theorem]{Proposition}
\newtheorem{conjecture}[theorem]{Conjecture}
\newtheorem{definition}[theorem]{Definition}

\theoremstyle{definition}

\begin{document}
\title{Numerical Semigroups of small and large type}
\author{Deepesh Singhal}
\address{Hong Kong University of Science and Technology}
\email{dsinghal@connect.ust.hk}

\date{June 2020}
\begin{abstract}
    A numerical semigroup is a sub-semigroup of the natural numbers that has a finite complement.
    Some of the key properties of a numerical semigroup are its Frobenius number F, genus g and type t.
    It is known that for any numerical semigroup $\frac{g}{F+1-g}\leq t\leq 2g-F$.
    Numerical semigroups with $t=2g-F$ are called almost symmetric, we introduce a new property that characterises them.
    We give an explicit characterisation of numerical semigroups with $t=\frac{g}{F+1-g}$.
    We show that for a fixed $\alpha$ the number of numerical semigroups with Frobenius number $F$ and type $F-\alpha$ is eventually constant for large $F$.
    Also the number of numerical semigroups with genus $g$ and type $g-\alpha$ is also eventually constant for large $g$.
\end{abstract}

\maketitle

\section{Introduction}

A numerical semigroup $S$ is a subset of natural numbers that contains $0$, is closed under addition and has a finite complement.
The numbers in the complement are called gaps, and the number of gaps is called the genus of $S$, is denoted by $g(S)$.
The largest gap is called the Frobenius number, is denoted by $F(S)$. The smallest nonzero element of $S$ is called its multiplicity and is denoted by $m(S)$.

A pseudo-Frobenius number of $S$ is a gap $x$ such that for any non-zero $s$ in $S$, $x+s$ is still in $S$. The collection of the pseudo-Frobenius numbers of $S$ is denoted by $PF(S)$, the number of pseudo-Frobenius numbers of $S$ is called its type and is denoted by $t(S)$.


For numerical semigroups in general it is known that
$$\frac{g(S)}{F(S)+1-g(S)}\leq t(S)\leq 2g(S)-F(S).$$
The first inequality was proved in \cite{Type}, and the second in \cite{H Nari}.
Numerical semigroups that satisfy $t(S)=2g(S)-F(S)$ are called almost symmetric, this notion was introduced in \cite{AS3} and has been studied in several papers \cite{H Nari,AS4,AS5,AS6,AS7,AS8,AS9}.
It is known that all irreducible numerical semigroups are almost symmetric.

For a numerical semigroup $S$ we define it's $T$-Set to be
$$T(S)=\{x\in\mathbb{N}| F(S)-x\in ((\mathbb{Z}\setminus S)\cup \{0\})\}.$$
Note that $T(S)$ is a subset of natural numbers, it contains $0$ and has a finite complement with respect to $\mathbb{N}$. However it is not necessarily closed under addition, such sets are called numerical sets.
In fact it is closed under addition pricisely when $S$ is almost symmetric, in Section \ref{T-Set} we prove that
\begin{theorem}\label{AS Characterisation}
A numerical semigroup is almost symmetric if and only if its $T$-Set is itself a numerical semigroup.
\end{theorem}

If $S$ is a numerical semigroup then $A^{\star}(S)=S\cup PF(S)$ is another numerical semigroup. In \cite{H Nari} it is proven that
\begin{theorem}\cite[Theorem 3.7]{H Nari}
Given an almost symmetric numerical semigroup $S$, $A^{\star}(S)$ is almost symmetric if and only if $t(A^{\star}(S))=m(S)-t(S)$.
\end{theorem}

By using $T$-Sets we find another equivalent condition for this to happen, we prove in Corollary \ref{S union F max ED} that
\begin{theorem}\label{sec 4 1}
If $S$ is an almost symmetric numerical semigroup then $A^{\star}(S)$ is almost symmetric if and only if $S\cup \{F(S)\}$ has max embedding dimension.
\end{theorem}

We also extend this result by applying the operation $A^{\star}$ multiple times. Note that $g(A^{\star}(S))<g(S)$ so this operation can only be applied finitely many times. We define $L(S)$ in Section \ref{Sec: T set of numerical set}, it measures how many times it can be applied. Let $m_i(S)$ be the $i^{th}$ positive element and $F_i(S)$ be the $i^{th}$ largest gap of $S$. We prove the following in Section \ref{Sec: T set of numerical set}.
\begin{theorem}\label{A star till k almost sym}
Given a numerical semigroup $S$, and $k< L(S)$.\\
$S,A^{\star}(S),\dots, A^{\star k}(S)$ are all almost symmetric numerical semigroups if and only if
\begin{itemize}
    \item $2i\in [0,k]$ implies $t(A^{\star 2i}(S))=2g(S)-m_{i}(S)-F_{i+1}(S)$
    \item $2i+1\in [0,k]$ implies $t(A^{\star 2i+1}(S))=F_{i+1}(S)-2g(S)+m_{i+1}(S)$
\end{itemize}
\end{theorem}

The set of all numerical semigroups can be arranged in a directed tree in which $S'$ is the parent of $S$ if $S'=S\cup\{F(S)\}$. This tree has been studied in many papers \cite{Semigroup Tree, Tree 1, Tree 2}.
\cite{Semigroup Tree} studies the infinite chains of this tree and characterises them.
The only numerical semigroups with $F(S)<m(S)$ are those of the form $\{0,f+1\rightarrow\}$, they form a chain and within this chain the type keeps increasing.
Note that in a chain of the tree if we have a numerical semigroup with $m(S)<F(S)$ then every numerical semigroup after that in the chain will also satisfy $m(S)<F(S)$.
We prove that if $m(S)<F(S)$ then
$$t(S\cup\{F(S)\})\geq t(S).$$
This implies that in any other chain the type of the numerical semigroup must eventually be constant. This is done in Section \ref{Sec: T set of numerical set}.
\begin{theorem}\label{eventually constant type}
Given a sequence $S_i$ of numerical semigroups such that for each $i$ $S_{i}=S_{i+1}\cup \{F(S_{i+1})\}$ and for at least one $i$ we have $m(S_i)<F(S_i)$. Then their type $t(S_i)$ will eventually be constant.
\end{theorem}

We also consider the numerical semigroups with the smallest type given $F,g$, in Section \ref{Sec: Staircase} we classify them.
\begin{theorem}\label{Staircase}
A numerical semigroup $S$ satisfies
$$t(S)=\frac{g(S)}{F(S)+1-g(S)}$$ if and only if either $S$ is symmetric or it is of the following form for some $m,n$
$$S=\{0,m,2m\dots nm\rightarrow\}.$$
\end{theorem}

Denote the number of numerical semigroups with Frobenius number $F$ and type $t$ by $T(F,t)$.
By $T_1(F,t)$ denote the number of almost symmetric numerical semigroups with Frobenius number $F$ and type $t$.
By $L(g,t)$ the number of numerical semigroups with genus $g$ and type $t$.

In \cite{Almost symmetric} it was proved that for a fixed $\alpha$, the function $T_1(F,F-2\alpha)$ remains constant once $F>4\alpha+1$.
This was done by finding a bijection between the set of almost symmetric numerical semigroups $S$ with $F(S)=F$, $t(S)=F-2\alpha$ and the set of numerical semigroups with genus $g$.
Following a similar method we show in Theorem \ref{F- alpha} and Theorem \ref{g- alpha} that

\begin{theorem}
For a fixed $\alpha$, the functions $T(F,F-\alpha)$ and $L(g,g-\alpha)$ are both eventually constant.
\end{theorem}

On the other hand for a fixed $\alpha$ the number of numerical semigroups of type $\alpha$ grow exponentially.
By constructing explicit families of numerical semigroups we show in Theorems \ref{construction AS type t} and \ref{construction type t} that
\begin{theorem}
For a fixed positive integer $\alpha$ and $F> 3\alpha+3$ such that $F\equiv \alpha (mod\;2)$
$$T_1(F,\alpha)\geq 2^{\frac{F}{6}-\frac{\alpha}{2}}.$$
Also for a fixed positive integer $\alpha$ and $\epsilon>0$, for sufficiently large $F$
$$T(F,\alpha)>2^{(0.1-\epsilon)F}.$$
\end{theorem}

\section{The T-Set of a numerical semigroup}\label{T-Set}

A numerical set is a subset of the natural numbers that contains $0$ and has a finite complement.
The genus (respectively Frobenius number) of a numerical set is defined to be the number of (respectively the largest) gaps.
The multiplicity of a numerical set is it's smallest non-zero element.
Every numerical set $T$ has an associated numerical semigroup given by $A(T)=\{x\mid \forall y\in T, x+y\in T\}$.
It is easily seen that $A(T)\subseteq T$ and they are equal if and only if $T$ is a numerical semigroup.

\begin{lemma}\label{f,g of T(S)}
For any numerical semigroup $S$, $F(T(S))=F(S)-m(S)$ and $g(T(S))=F(S)-g(S)$.
\end{lemma}
\begin{proof}
Notice that the gaps of $T(S)$ are $F(S)-s$ for $s$ in $S\cap [1, F(S)]$.
\end{proof}

\begin{theorem}\label{A of T of S}
Let $S$ be a numerical semigroup, then 
$$A(T(S))=S\cup PF(S).$$
\end{theorem}
\begin{proof}
If $x$ is not in $S\cup PF(S)$, then there is a $y\in S$, $y\neq 0$ such that $x+y\not\in S$.
Then $F(S)-x-y\in T(S)$, $F(S)-y\not\in T(S)$ and hence $x\not\in A(T(S))$.

Now for the other direction, if $x$ is not in $A(T(S))$ then there is a $y\in T(S)$ such that $x+y\not\in T(S)$.
Then $F(S)-y\in ((\mathbb{Z}\setminus S)\cup \{0\})$, $F(S)-x-y\in (S\setminus \{0\})$.
It follows that $x\not\in S\cup PF(S)$.
\end{proof}
\begin{corollary}\label{genus of T of S}
$t(S)=g(S)-g(A(T(S))).$
\end{corollary}

\begin{proof}[Proof of Theorem \ref{AS Characterisation}]
By Corollary \ref{genus of T of S} and Lemma \ref{f,g of T(S)}
$$2g(S)-F(S)-t(S)
=g(S)-F(S)+g(A(T(S)))
=g(A(T(S)))-g(T(S)).$$
Therefore, $t(S)=2g(S)-F(S)$ if and only if $g(A(T(S)))=g(T(S))$ which happens if and only if $T(S)$ is a numerical semigroup.
\end{proof}

For a numerical semigroup $S$, the quantity $2g(S)-F(S)-t(S)$ measures how far $S$ is from being almost symmetric, this is equal to $g(A(T(S)))-g(T(S))$ which measures how far the $T$-set of $S$ is from being a numerical semigroup.

\begin{corollary}
A numerical semigroup $S$ is almost symmetric if and only if it satisfies the property that whenever $x,y\in Gap(S)$ and $x+y>F(S)$ we have $x+y-F(S)\in Gap(S)$.
\end{corollary}
\begin{proof}
Consider $a,b\in T(S)$, such that $a,b<F(S)$.
Say $a=F(S)-x$, $b=F(S)-y$, so $x,y\in Gap(S)$.
Now if $x+y\leq F(S)$ then $a+b=2F(S)-x-y\geq F(S)>F(T(S))$ and hence $a+b\in T(S)$.
On the other hand if $x+y>F(S)$ then $a+b\in T(S)$ if and only if $x+y-F(S)\in Gap(S)$.
\end{proof}

\section{Numerical Semigroups of type $F-\alpha$, $g-\alpha$}
In this section we will show that for a fixed $\alpha$ the functions $T(F,F-\alpha)$ and $L(g,g-\alpha)$ are both eventually constant.
We will do this by finding a bijection between numerical semigroups with Frobenius number $F$, type $F-\alpha$ and numerical sets $T$ with $g(T)+g(A(T))=\alpha$.
And for $L(g,g-\alpha)$ we find a bijection between numerical semigroups with genus $g$, type $g-\alpha$ and numerical sets $T$ with $g(A(T))=\alpha$.

We use the notation that for a set $X$, $a-X=\{a-x\mid x\in X\}$.

\begin{lemma}\label{Construct S}
Given a numerical set $T$ with $F(T)=f$ and given an integer $F$ larger than $2f$. Consider
$$S=(\mathbb{N}\setminus (F-T))\cup \{0\}$$
Then $S$ is a numerical semigroup with Frobenius number $F$, $T(S)=T$ and type $F-g(T)-g(A(T))$.
\end{lemma}
\begin{proof}
It is clear that $S$ has Frobenius number $F$.
Next, we want to show that $S$ is closed under addition. Consider $x,y$ in $S$, assume that they are non-zero and $x,y<F$. Then $F-x, F-y$ are gaps of $T$ which implies $F-x, F-y\leq f$ and hence $x,y\geq F-f>\frac{F}{2}$. Therefore $x+y>F$, $x+y\in S$. Therefore $S$ is a numerical semigroup.

It follows from the definition of $T$-Set that $T(S)=T$. Corollary \ref{genus of T of S} and Lemma \ref{f,g of T(S)} imply that 
$$t(S)=g(S)-g(A(T))=F-g(T)-g(A(T)).$$
\end{proof}

\begin{theorem}\label{F- alpha}
Given positive integers $\alpha, F$ such that $F>4\alpha-6$. There is a bijection between the set of numerical semigroups with Frobenius number $F$, type $F-\alpha$ and the set of numerical sets $T$ for which $g(T)+g(A(T))=\alpha$.
\end{theorem}
\begin{proof}
Given a numerical semigroup of Frobenius number $F$ and type $F-\alpha$. We know by Corollary \ref{genus of T of S} and Lemma \ref{f,g of T(S)} that,
$$\alpha=F-t(S)
=F-(g(S)-g(A(T(S))))
=g(T(S))+g(A(T(S))).$$
Since the Frobenius number $F$ is fixed, the map $T(S)$ is injective.

Next, we need to show that it is surjective. Consider a numerical set $T$ with $g(T)+g(A(T))=\alpha$. Then $g(A(T))\leq\alpha-1$ and hence 
$$F(T)=F(A(T))\leq 2g(A(T))-1\leq 2\alpha-3.$$
Therefore $2F(T)\leq 4\alpha-6<F$, and hence Lemma \ref{Construct S} gives a numerical semigroup $S$ with Frobenius number $F$, type $F-\alpha$ and $T(S)=T$.
It follows that the map is surjective and hence bijective.
\end{proof}

\begin{corollary}\label{T_1}
(Proved in \cite{Almost symmetric})
Given positive integers $F,\beta$ such that $F>8\beta-6$. There is a bijection between the set of almost symmetric numerical semigroups with Frobenius number $F$ and type $F-2\beta$ and the number of numerical semigroups with genus $\beta$.
\end{corollary}
\begin{proof}
Let $\alpha=2\beta$, restrict the map to almost symmetric numerical semigroups and use the fact that $S$ is almost symmetric if and only if $T(S)$ is a numerical semigroup (Theorem \ref{AS Characterisation}).
\end{proof}

Next we will obtain a similar bijection for the collection of numerical semigroups with genus $g$ and type $g-\alpha$. But first we need a result from \cite{My SDSU}. Given a numerical semigroup $S$ they define
$$B(S)=\{x\mid x\not\in S, F(S)-x\not\in S\}.$$
Notice that $|B(S)|=2g(S)-F(S)-1$.

\begin{theorem}\label{TBUS}
Given a numerical semigroup $S$, numerical set $T$ such that $A(T)=S$, $T$ must satisfy
$$S\subseteq T\subseteq S\cup B(S).$$
\end{theorem}
\begin{proof}
See \cite{My SDSU}.
\end{proof}
\begin{corollary}\label{cor to TBUS}
For any numerical set $T$
$$g(T)\geq F(T)-g(A(T))+1$$
\end{corollary}
\begin{proof}
Theorem \ref{TBUS} implies that
$$g(T)\geq g(A(T))-|B(A(T))|=g(A(T))-(2g(A(T))-F(A(T))-1).$$
Also $F(T)=F(A(T))$, so we are done.
\end{proof}

\begin{theorem}\label{g- alpha}
Given positive integers $\alpha, g$ such that $g\geq 3\alpha-1$. There is a bijection between the set of numerical semigroups with genus $g$, type $g-\alpha$ and the set of numerical sets $T$ for which $g(A(T))=\alpha$.
\end{theorem}
\begin{proof}
Given a numerical semigroup $S$ with genus $g$ and type $g-\alpha$, we know by Corollary \ref{genus of T of S} that $g(A(T(S)))=\alpha$.
We first show that this map is injective. Given two such $S_1,S_2$, if $T(S_1)=T(S_2)$ by Lemma \ref{f,g of T(S)} we have
$$F(S_1)=g(S_1)+g(T(S_1))
=g(S_2)+g(T(S_2))=F(S_2).$$
Since $S_1,S_2$ have the same Frobenius number and $T$-Sets it follows that $S_1=S_2$.

Next we show that the map is surjective. Consider a numerical set $T$ such that $g(A(T))=\alpha$. Let $F=g+g(T)$. We know that
$$F(T)=F(A(T))\leq 2g(A(T))-1=2\alpha-1.$$
Therefore by Corollary \ref{cor to TBUS}
$$F\geq 3\alpha-1+g(T)
\geq 3\alpha-1+F(T)-\alpha+1
= F(T)+2\alpha>2F(T).$$
Now Lemma \ref{Construct S} gives us a numerical semigroup $S$ with Frobenius number $F$ and $T(S)=T$. Finally lemma \ref{f,g of T(S)} implies that $g(S)=F-g(T)=g$ and Corollary \ref{genus of T of S} implies that $t(S)=g-g(A(T))=g-\alpha$. Therefore, the map is surjective and hence bijective.
\end{proof}

\section{The type and $T$-Set of a numerical set}\label{Sec: T set of numerical set}

In this section we will define the $T$-Set for a numerical set and study it's properties proving several interesting results like Theorem \ref{sec 4 1}, Theorem \ref{A star till k almost sym} and Theorem \ref{eventually constant type} on the way.
One of the ways of obtaining a numerical semigroup from a numerical set is the associated semigroup $A(T)$.
There is another useful way of obtaining a numerical semigroup from a numerical set.
Given a numerical set $H\neq \mathbb{N}$ we define
$$A^{\star}(H)=\{a\mid a+(H\setminus\{0\})\subseteq H\}$$
We make a note of some observations, given a numerical set $H\neq \mathbb{N}$
\begin{itemize}
    \item $A^{\star}(H)$ is a numerical semigroup
    \item $A(H)= A^{\star}(H)\cap H$
    \item If $H$ is itself a numerical semigroup then $A^{\star}(H)=H\cup PF(H)$
\end{itemize}
We now define the pseudo-Frobenius numbers and type for numerical sets analogously to numerical semigroups.
Our definition of type of a numerical set is different from the one given in \cite{Type 2}.
For a numerical set $H$ we define
$$PF(H)=\{a\in\mathbb{N}\mid a\not\in H, a+(H\setminus\{0\})\subseteq H\}$$
and the type of $H$ is $t(H)=|PF(H)|$.
It is clear that $F(H)\in PF(H)$, so $t(H)\geq 1$. Also
$$A^{\star}(H)=PF(H)\dot{\cup}A(H).$$

\begin{lemma}
For any numerical set $H$, $PF(H)\subseteq PF(A(H))$ and therefore $t(H)\leq t(A(H))$.
\end{lemma}
\begin{proof}
Given $P\in PF(H)$, we know that $P\notin H$ therefore $P\notin A(H)$. If $x\in A(H)$, $x\neq 0$ then we want to show that $P+x\in A(H)$. Pick $h\in H$, if $h\neq 0$ then $P+h\in H$ and hence $P+h+x\in H$. On the other hand if $h=0$, then $P+x\in H$ as $x\in H\setminus\{0\}$. Therefore $P+x\in A(H)$, which in turn implies $P\in PF(A(H))$.
\end{proof}

\begin{definition}
Let $H\neq \mathbb{N}$ be a numerical set with Frobenius number $F$, then it's $T$-Set is defined to be the following numerical set
$$T(H)=\{x\in\mathbb{N}| F-x\in ((\mathbb{Z}\setminus H)\cup \{0\})\}.$$
\end{definition}

We now find the analogues of the results from Section \ref{T-Set}.

\begin{lemma}\label{F,g of T(H)}
If $H\neq\mathbb{N}$ is a numerical set then $F(T(H))=F(H)-m(H)$, $g(T(H))=F(H)-g(H)$.
\end{lemma}
\begin{proof}
Note that the gaps of $T(H)$ are precisely $F(H)-x$ for $x$ in $H\cap [1,F(H)]$.
\end{proof}

\begin{proposition}
For a numerical set $H\neq \mathbb{N}$, $A(T(H))=A^{\star}(H)$.
\end{proposition}
\begin{proof}
Pick a natural number $a$.
We know that $a$ is not in $A(T(H))$ if and only if there is a $x\in T(H)$ such that $x+a\not\in T(H)$. This means there is an $x$ such that $F(H)-x\in (\mathbb{Z}\setminus H)\cup \{0\}$ and $F(H)-x-a\in H\setminus\{0\}$. This is equivalent to saying that there is a $y\in H\setminus\{0\}$ such that $y+a\not\in H$ i.e. $a\not\in A^{\star}(H)$.
It follows that $A(T(H))=A^{\star}(H)$.
\end{proof}
\begin{corollary}\label{t(H))}
$t(H)=g(A(H)))-g(A(T(H)))$.
\end{corollary}

\begin{theorem}\label{T(H) is NS}
Given a numerical set $H\neq\mathbb{N}$, we have
$$t(H)\leq g(H)+g(A(H))-F(H).$$
Moreover equality holds if and only if $T(H)$ is a numerical semigroup.
\end{theorem}
\begin{proof}
We know by Lemma \ref{F,g of T(H)} and Corollary \ref{t(H))} that
$$g(A(T(H)))-g(T(H))
= g(A(H))-t(H)-(F(H)-g(H))$$
$$=g(A(H))+g(H)-F(H)-t(H).$$
This quantity must be non-negative and it is zero if and only if $T(H)$ is a numerical semigroup.
\end{proof}

A numerical set $H$ is called ordinary if $m(H)=F(H)+1$ i.e. $H=\{0,F(H)+1\rightarrow\}$. Note that $T(H)=\mathbb{N}$ if and only if $H$ is ordinary.
Since we have defined the $T$-Set for an arbitrary numerical set, we can apply in multiple times.

\begin{theorem}
For a non-ordinary numerical set $H$,
$$t(T(H))+t(H)\leq g(A(H))-g(H)+m(H).$$
Moreover, equality holds if and only if $T(T(S))$ is a numerical semigroup.
\end{theorem}
\begin{proof}
We apply Theorem \ref{T(H) is NS},
\begin{align*}
&\ g(T(H))+g(A(T(H)))-F(T(H))\\
&\ =F(H)-g(H)+g(A(H))-t(H)-(F(H)-m(H))\\
&\ =g(A(H))-g(H)+m(H)-t(H).
\end{align*}
\end{proof}

\begin{corollary}\label{T^2(S) is NS t+t=m}
If $S$ is a non-ordinary numerical semigroup then
$T(T(S))$ is a numerical semigroup if and only if $t(S)+t(T(S))=m(H)$.
\end{corollary}

\begin{corollary}[From \cite{H Nari}]
If $S$ is an almost symmetric numerical semigroup then $A^{\star}(S)$ is almost symmetric if and only if $t(S)+t(A^{\star}(S))=m(S)$.
\end{corollary}

The transformation obtained by applying the $T$-Set twice has a simpler description as follows.

\begin{lemma}\label{shift T^2}
For a non-ordinary numerical set $H$,
$$T(T(H))=\{x-m(H)\mid x\in H\cup\{F(H)\}, x\neq 0\}$$
\end{lemma}
\begin{proof}
We know that $T(T(H))$ consists of elements of the form $F(T(H))-x$ for $x\in Gap(T(H))$ along with all natural numbers $\geq F(T(H))$. Also we know that $F(T(H))=F(H)-m(H)$ and the gaps of $T(H)$ are $F(H)-y$ for $y$ in $T\cap [1,F(H)]$. The result follows.
\end{proof}

Lemma \ref{shift T^2} implies that the gaps of $T(T(H))$ are $x-m(H)$ for $x$ in $(Gap(H)\setminus\{F(H)\})\cap [m(H)+1, F(H)]$. And therefore
$$g(T(T(H)))=g(H)-m(H).$$

\begin{proposition}\label{A of T^2 H}
$A(T(T(H)))=A^{\star}(H\cup \{F(H)\})$.
\end{proposition}
\begin{proof}
Pick a natural number $a$.
By Lemma \ref{shift T^2} $a\in A(T(T(H)))$ if and only if $x\in H\cup\{F(H)\}\setminus\{0\}$ implies $a+x\in H\cup\{F(H)\}\setminus\{0\}$. Which happens if and only if $a\in A^{\star}(H\cup\{F(H)\})$.
\end{proof}
\begin{corollary}\label{type decresing}
If $S$ is a non-ordinary numerical semigroup then
$$t(S\cup\{F(S)\})=t(S)+t(T(S))-1\geq t(S).$$
\end{corollary}
\begin{proof}
By applying Corollary \ref{t(H))} twice we see that
$$g(A(T(T(H))))=g(A(T(H)))-t(T(H))
=g(A(H))-t(H)-t(T(H)).$$
We now use the above proposition.
$$t(S\cup\{F(S)\})
=g(S\cup\{F(S)\})-g(A^{\star}(S\cup\{F(S)\}))$$
$$=g(S)-1-g(A(T(T(S))))
=g(S)-1-(g(S)-t(S)-t(T(S))).$$
\end{proof}
\begin{corollary}
If $S$ is a non-ordinary numerical semigroup then
$T(T(S))$ is a numerical semigroup if and only if $S\cup \{F(S)\}$ has max embedding dimension.
\end{corollary}
\begin{proof}
It is a known fact that a numerical semigroup has max embedding dimension if and only if its type is one less than the multiplicity.
$S$ is not ordinary so $m(S\cup\{F(S)\})=m(S)$. It follows that $S\cup\{F(S)\}$ is max embedding dimension if and only if $m(S)=t(S)+t(T(S))$. By Corollary \ref{T^2(S) is NS t+t=m} this happens if and only if $T(T(S))$ is a numerical semigroup.
\end{proof}
\begin{corollary}\label{S union F max ED}
If $S$ is an almost symmetric numerical semigroup then $A^{\star}(S)$ is almost symmetric if and only if $S\cup \{F(S)\}$ has max embedding dimension.
\end{corollary}

\begin{proof}[Proof of Theorem \ref{eventually constant type}]
If $S_n$ is not ordinary it follows that none of the $S_i$ are ordinary for $i\geq n$. Therefore by Corollary \ref{type decresing} the sequence $t(S_{n+i})$ is non increasing and positive which implies that it must be eventually constant.
\end{proof}

So far we have applied the $T$-Set operation twice, we now look at how many times it can be applied to a given numerical set. Given a numerical set $H$, define $l(H)=|H\cap [1,g(H)]|$. Moreover, if $g(H)\in H$ then define  $L(H)=2l(H)$.
And if $g(H)\not\in H$ then define $L(H)=2l(H)+1$.
Note that $L(H)=0$ if and only if $H=\mathbb{N}$, so $T(H)$ is defined if and only if $L(H)\geq 1$.

\begin{lemma}
Given a numerical set $H\neq \mathbb{N}$, $L(T(H))=L(H)-1$.
\end{lemma}
\begin{proof}
Firstly notice that
$$|Gap(H)\cap [g(H)+1,F(H)]|
=g(H)-|Gap(H)\cap [1,g(H)]|
=l(H).$$
Therefore
\begin{align*}
     l(T(H))=&\ |T(H)\cap [1,F(H)-g(H)]|=|Gap(H)\cap [g(H),F(H)-1)|\\
    &\ =l(H)-1+|Gap(H)\cap \{g(H)\}|
\end{align*}
Now if $g(H)\in H$ then $g(T(H))=F(H)-g(H)\not\in T(H)$.
In this case $l(T(H))=l(H)-1$ and hence $L(T(H))=L(H)-1$.

On the other hand if $g(H)\not\in H$ then $g(T(H))=F(H)-g(H)\in T(H)$.
In this case $l(T(H))=l(H)$ and $L(T(H))=L(H)-1$.
\end{proof}

It follows that given a numerical set $H$, we can construct
$$H,T(H), T^2(H),\dots, T^{L(H)}(H)=\mathbb{N}.$$
Now define $m_i(H)$ to be the $i^{th}$ positive element of $H$. Let $m_0(H)=0$.
For $i\leq g(H)$ let $F_i(H)$ be the $i^{th}$ largest gap of $H$.

\begin{proposition}\label{shift mi}
Given a numerical set $H$ and positive integer $i$, such that $2i\leq L(H)$. We have
$$T^{2i}(H)=\{x-m_{i}(H)\mid x\geq m_i(H), x\in H\cup \{F_1(H),F_2(H),\dots,F_i(H)\}\}.$$
\end{proposition}
\begin{proof}
We induct on $i$, the base case $i=1$ is already known from Lemma \ref{shift T^2}. Now assume this holds for some $i$, and $2i+2\leq L(H)$. We write $m_j$ for $m_j(H)$ and $F_j$ for $F_j(H)$.
\begin{align*}
    &\ T^{2(i+1)}(H)=T^{2i}(T^{2}(H))\\
    &\ =\{x-m_{i}(T^2(H))\mid x\geq m_i(T^2(H)), x\in T^2(H)\cup \{F_j(T^2(H))\mid 1\leq j\leq i\}\}\\
    &\ =\big\{x-(m_{i+1}-m)\mid x\geq m_{i+1}-m, x+m\in H\cup\{F(H)\}\cup \{F_j\mid 2\leq j\leq i+1\}\big\} \\
    &\ =\{y-m_{i+1}\mid y\geq m_{i+1}, y\in H\cup \{F_1,F_2,\dots,F_{i+1}\}\}.
\end{align*}
\end{proof}

We therefore see that if $2i\leq L(H)$ then
$F(T^{2i}(H))=F_{i+1}-m_{i}$,
$g(T^{2i}(H))=g(H)-m_{i}$ and
$m(T^{2i}(H))=m_{i+1}-m_{i}$.
And hence if $2i+1\leq L(H)$ then
$F(T^{2i+1}(H))=F_{i+1}-m_{i+1}$,
$g(T^{2i+1}(H))=F_{i+1}-g(H)$ and
$m(T^{2i+1}(H))=F_{i+1}-F_{i+2}$.

\begin{proposition}
Given $k<L(H)$
$$\sum_{i=0}^{k}t(T^{i}(H))=g(A(H))-g(A(T^{k+1}(H))).$$
In particular
$$g(A(H))=\sum_{i=0}^{L(H)-1}t(T^{i}(H)).$$
\end{proposition}
\begin{proof}
This follows from the fact that
$t(T^i(H))=g(A(T^i(H)))-g(A(T^{i+1}(H)))$.
\end{proof}
\begin{corollary}\label{Even}
If $2i\leq L(H)$ then
$$\sum_{j=0}^{2i-1} t(T^{j}(H))\leq g(A(H))-g(H)+m_i(H).$$
Moreover equality holds if and only if $T^{2i}(H)$ is a numerical semigroup.
\end{corollary}
\begin{proof}
We use the fact that $g(A(T^{2i}(H)))\geq g(T^{2i}(H))=g(H)-m_{i}(H)$. And equality holds if and only if $T^{2i}(H)$ is a numerical semigroup.
\end{proof}
\begin{corollary}\label{Odd}
If $2i+1\leq L(H)$ then
$$\sum_{j=0}^{2i} t(T^{j}(H))\leq g(A(H))+g(H)-F_{i+1}(H).$$
Moreover equality holds if and only if $T^{2i+1}(H)$ is a numerical semigroup.
\end{corollary}
\begin{proof}
This is similar, $g(T^{2i+1}(H))=F_{i+1}(H)-g(H)$.
\end{proof}

\begin{proof}[Proof of Theorem \ref{A star till k almost sym}]
If they are all almost symmetric then for each $j\in [0,k]$ $A^{\star j}(S)=T^{j}(S)$, and hence $T(A^{\star j}(S))=T^{j+1}(S)$ is a numerical semigroup. The result follows from Corollary \ref{Even} and Corollary \ref{Odd}.

Conversely if the $t(A^{\star j}(S))$ are given by these expressions then $t(S)=2g(S)-m_{0}(S)-F_{1}(S)=2g(S)-F(S)$. Therefore $S$ is almost symmetric which implies that $T(S)=A^{\star}(S)$. Continuing by induction we see that for each $j\in [0,k]$ $T^{j}(S)=A^{\star j}(S)$ and it is almost symmetric.
\end{proof}

For a non-ordinary numerical set $H$, it's ordinarisation transforsm is
$O(H)=H\cup\{F(H)\}\setminus\{m(H)\}$.
This was defined in \cite{Ordinarization transform}.
They also defined $l(S)$ for a numerical semigroup $S$ and called it the ordinarizaation number of $S$.

\begin{proposition}
For $2\leq 2i\leq L(H)$,
$$A(T^{2i}(H))=A^{\star}(O^{i-1}(H\cup\{F(H)\})).$$
Moreover for $0\leq 2i<L(H)$,
$$A(T^{2i+1}(H))=A^{\star}(O^{i}(H)).$$
\end{proposition}
\begin{proof}
Proposition \ref{shift mi} implies that $T(T^{2i}(H))=T(O^{i}(H))$. Therefore we have $A(T^{2i+1}(H))=A^{\star}(O^{i}(H))$.
$A(T^{2i}(H))$ is computed directly using Proposition \ref{shift mi} analogously to Proposition \ref{A of T^2 H}.
\end{proof}

\begin{corollary}
If $S$ is a numerical semigroup then
$$t(O^{i}(S))=\sum_{j=0}^{2i}t(T^{j}(S)) \leq 2g(S)-F_{i+1}(S).$$
Equality holds if and only if $T^{2i+1}(S)$ is a numerical semigroup which happens if and only if $O^{i}(S)$ is almost symmetric.
\end{corollary}
\begin{proof}
We know that
$$t(O^{i}(S))=g(O^{i}(S))-g(A^{\star}(O^{i}(S)))
=g(S)-g(A(T^{2i+1}(S)))
=\sum_{j=0}^{2i}t(T^{j}(S)).$$
The result follows from Corollary \ref{Odd} and the fact that $F(O^{i}(S))=F_{i+1}(S)$.
\end{proof}

\section{Numerical Semigroups of small type}\label{Sec: Staircase}
We define $n(S)$ to be the size of $S\cap [0, F(S)]$, i.e. $n(S)=F(S)+1-g(S)$.
In their Theorem 20 the authors of \cite{Type} prove that
$$ \frac{g(S)}{n(S)}\leq t(S).$$
In this section we state a slight variation of their proof and then classify the numerical semigroups for which equality holds.
Given a numerical semigroup $S$ define a function $\phi_S$ from $Gap(S)$ to $S$,
$$\phi_S (x)=Max\{s\in S| x+s\in Gap(S)\}.$$

\begin{lemma}
For any numerical semigroup $S$ and $x\in Gap(S)$
\begin{enumerate}
    \item $0\leq \phi_S(x)<F(S)$
    \item $\phi_S(x)=0$ if and only if $x\in PF(S)$
    \item $x+\phi_S(x)\in PF(S)$
\end{enumerate}
\end{lemma}
\begin{proof}
The first part is clear, the second is essentially a restatement of the definition of a pseudo-Frobenius number.
For the third part, given $s\in S$ with $s\neq 0$, we know that $\phi_S(x)+s$ is in $S$ and it is bigger than $\phi_S(x)$. Therefore, the maximality of $\phi_S(x)$ implies that $x+\phi_S(x)+s$ is in $S$. Hence, $x+\phi_S(x)$ is a pseudo-Frobenius number of $S$.
\end{proof}

\begin{theorem}[From \cite{Type}]\label{g<nt}
Given a numerical semigroup $S$,
$$g(S)\leq n(S)t(S).$$
\end{theorem}
\begin{proof}
Consider the map $\psi_S$ from $Gap(S)$ to $\left( S\cap [0,F]\right)\times PF(S)$, $\psi_S(x)=(\phi(x),x+\phi(x))$. It is clear that $\psi_S$ is injective and hence $g(S)\leq n(S)t(S)$.
\end{proof}

\begin{lemma}\label{F-m}
Given a numerical semigroup $S$ with Frobenius number $F$ and multiplicity $m$, $[F-m+1, F]$ is a subset of $S\cup PF(S)$.
\end{lemma}
\begin{proof}
By Theorem \ref{A of T of S} we know that $S\cup PF(S)=A(T(S))$. Moreover $F(A(T(S)))=F(T(S))=F-m$ by Lemma \ref{f,g of T(S)}. The result follows.
\end{proof}

\begin{proof}[Proof of Theorem \ref{Staircase}]
First, if $S$ is symmetric then $t(S)=1$, $g(S)=n(S)$ so the equation is satisfied.
Next, for $S=\{0,m,2m\dots nm\rightarrow\}$, $F(S)=nm-1$, $n(S)=n$ and $g(S)=F(S)+1-n(S)=n(m-1)$.
Also, $PF(S)=[(n-1)m+1,nm-1]$ so $t(S)=m-1$ and $g(S)=n(S)t(S)$.

Now, assume that $S$ is a numerical semigroup that satisfies $g(S)=n(S)t(S)$, is not symmetric. Therefore, $t(S)\geq 2$ and $m(S)\geq t(S)+1$. Now, $g(S)=n(S)t(S)$ implies that the map $\psi_S$ from proof of Theorem \ref{g<nt} must be surjective. Therefore, for any $s\in S\cap [1, F(S)]$ and $P\in PF(S)$, $s<P$. Moreover, the map $\lambda_S$ from
$\left( S\cap [0,F]\right)\times PF(S)$
to $Gap(S)$, $\lambda_S(s,P)=P-s$ is injective as it is the inverse of $\psi_S$.

By Lemma \ref{F-m} the interval $[F-(m-1),F]$ consists entirely of elements of $S$ and its pseudo-Frobenius numbers.
It has $m$ consecutive numbers and hence contains a multiple of $m$ which will be in $S$.
This tells us that $PF(S)=[F-(t-1),F]$ and $[F-(m-1), F-t]\subseteq S$.

If $m\geq t+2$ then $F-(t+1)$ will be in $S$. But then
$$\lambda_S(F-(t+1),F-(t-1))=2=\lambda_S(F-t,F-(t-2))$$
which contradicts the fact that $\lambda_S$ is injective. Therefore, $m=t+1$.

Now since every number in the interval $[F-(m-2), F]$ is a gap (in fact a pseudo-Frobenius number) none of them can be a multiple of $m$. Therefore, $F\equiv -1(mod\; m)$ say $F=nm-1$.

Finally if any $s\in S\cap [0,F]$ is not a multiple of $m$, then say $s+r\equiv -1 (mod\;m)$ with $0\leq r\leq m-2$. Then $F-r\equiv s(mod\;m)$ and $F-r\geq F-(m-2)> s$ i.e. $F-r=s+km$ for some $k\geq 1$, but this implies that $F-r\in S$ which is a contradiction. Therefore, $S\cap [0, F]$ consists entirely of multiples of $m$ and hence $S=\{0,m,2m\dots nm\rightarrow\}$.
\end{proof}

\section{Numerical semigroups of type $\alpha$}
In this section we will give lower bounds for number of numerical semigroups of given type and Frobenius number. We will do so by constructing explicit families of numerical semigroups.
Numerical semigroups of type $1$ always have an odd Frobenius number. The authors of \cite{Type} showed that for odd $F$ the number of type $1$ numerical semigroups grows exponentially and is at least $2^{\left\lfloor\frac{F}{8}\right\rfloor}$.
In \cite{Backelin} the following theorem is proved. 

\begin{theorem}
For $i\in\{1,3,5\}$ the following limit exist and are positive
$$\lim_{F\equiv i(mod\;6), F\to\infty} \frac{T(F,1)} {2^{\frac{F}{6}}}.$$
\end{theorem}
\begin{proof}
See \cite{Backelin}.
\end{proof}

We will now look at type at least $2$ and start with constructing a family of almost symmetric numerical semigroups of a given type.

\begin{theorem}\label{construction AS type t}
Assume $F>6k+6$, let $A$ be subset of $(\frac{F}{3}, \lfloor\frac{F-1}{2}\rfloor -k)$
$$B=\left\{x\mid \left\lceil\frac{F+1}{2}\right\rceil+k<x<F, F-x\not\in A \right\}$$
$$S=\{0\}\cup A\cup B\cup \{F+1\rightarrow\}$$
Then $S$ is an almost symmetric numerical semigroup with Frobenius number $F$.
Moreover if $F$ is odd then $t(S)=2k+1$ and if $F$ is even then $t(S)=2k+2$.
\end{theorem}
\begin{proof}
We first show that $S$ is a numerical semigroup for which we need to check that $S$ is closed under addition. Consider $x,y\in S$, assume that $x,y\neq 0$ and $x+y\leq F$ as otherwise we have nothing to prove. Also without loss of generality assume $x\leq y$, which implies that $x<\frac{F}{2}$ and hence $x\in A$. We know that $x+y>\frac{2F}{3}>\left\lceil\frac{F+1}{2}\right\rceil+k$, which implies that $F-(x+y)<\frac{F}{3}$ and hence $F-(x+y)\not\in A$. We know $x+y$ cannot equal $F$, so $x+y<F$. It follows that $x+y\in B$.

Now, given a number $x$ such that $\lfloor\frac{F-1}{2}\rfloor -k\leq x\leq \left\lceil\frac{F+1}{2}\right\rceil+k$, we firstly know that $x$ is not in $S$. We will show that $x$ is a pseudo-Frobenius number of $S$.
Consider for a non-zero $s$ in $S$, assume that $x+s\leq F$ as otherwise we have nothing to prove.
Note that $x+s$ cannot be $F$, so $x+s<F$.
Moreover,
$$x+s>\frac{F}{3}+ \left\lfloor\frac{F-1}{2}\right\rfloor -k\geq \frac{F}{3}+\frac{F}{2}-1-k>\frac{5F}{6}-\frac{F}{6}=\frac{2F}{3}$$
So, $F-(x+s)<\frac{F}{3}$ and $F-(x+s)\not\in A$.
Also it can be checked that $F>6k+6$ implies that $\left\lceil\frac{F+1}{2}\right\rceil+k< \frac{2F}{3}$.
Therefore, $\left\lceil\frac{F+1}{2}\right\rceil+k<x+s<F$ and hence $x+s\in B$. This shows that $x$ is a pseudo-Frobenius number of $S$.

Next, for any other $x$ in $[1, F-1]$, either $x\in S$ or $F-x\in S$, either way $x$ cannot be a pseudo-Frobenius number. Therefore 
$$PF(S)=\left[\left\lfloor\frac{F-1}{2}\right\rfloor -k, \left\lceil\frac{F+1}{2}\right\rceil+k\right]\cup\{0\}$$
It is now clear that if $F$ is odd then $t(S)=2k+1$ and if $F$ is even then $t(S)=2k+2$.
\end{proof}
\begin{corollary}\label{AS lower bound}
For a fixed positive integer $\alpha$ and $F> 3\alpha+3$ such that $F\equiv \alpha (mod\;2)$
$$T_1(F,\alpha)\geq 2^{\frac{F}{6}-\frac{\alpha}{2}}.$$
\end{corollary}

\begin{theorem}\label{construction type t}
Pick an irrational $\beta$ such that $\frac{2}{5}<\beta<\frac{1}{2}$.
And for $F>\text{Max }(\frac{k+1}{5\beta -2},\frac{k}{1-2\beta})$, pick a subset $A$ of $(\beta F,\frac{F}{2})$.
Let
$$B=\left\{x\mid \frac{F}{2}<x< (1-\beta)F, F-x\not\in A\right\},$$
$S=\{0\}\cup A\cup B
\cup \big[\lceil(1-\beta)F\rceil, \lfloor2\beta F\rfloor-k\big]
\cup \big[\lceil 2\beta F\rceil,F-1\big]
\cup \{F+1\rightarrow\}.$
Then $S$ is a numerical semigroup of Frobenius number $F$. If $F$ is odd then $t(S)=k+1$ and if $F$ is even then $t(S)=k+2$.
\end{theorem}
\begin{proof}
We show that $S$ is closed under addition, consider $x,y\in S$, assume that $x,y\neq 0$ and $x+y\leq F$ as otherwise we have nothing to prove. Clearly $x+y\neq F$, so $x+y<F$. 
We know that $x,y>\beta F$ and therefore $x+y>2\beta F$ which means that $x+y\in \big[\lceil 2\beta F\rceil,F-1\big]$.

Now for $0\leq i\leq k-1$, $\lfloor2\beta F\rfloor-i\not\in S$. Moreover for any non-zero $s\in S$, $s> \beta F$ and hence $\lfloor2\beta F\rfloor-i+s>2\beta F-k+\beta F>F$. This means that $\lfloor2\beta F\rfloor-i$ is a pseudo-Frobenius number of $S$.

Let $n=F-(\lfloor2\beta F\rfloor-i)$ and let $s=\lfloor2\beta F\rfloor-n
=2\lfloor2\beta F\rfloor-F-i$.
This means that $s>(4\beta-1)F-(k-1)-2>(1-\beta)F$.
Also $s<(4\beta-1)F< 2\beta -k$, which means $s\in \big[\lceil(1-\beta)F\rceil, \lfloor2\beta F\rfloor-k\big]$. So we have $s\in S$, $s\neq 0$ and $n+s=\lfloor2\beta F\rfloor\not\in S$. Therefore, $n$ is not a pseudo-Frobenius number of $S$.

If $F$ is even, then for any non-zero $s\in S$, $s>\beta F$ so $\frac{F}{2}+s>(\frac{1}{2}+\beta)F>2\beta F$. Clearly $\frac{F}{2}+s\neq F$, so $\frac{F}{2}+s\in S$. Therefore $\frac{F}{2}$ is a pseudo-Frobenius number of $S$.

Finally for any other $x$ in $[1, F-1]$ either $x\in S$ or $F-x\in S$, which means that $x$ cannot be a pseudo-Frobenius number of $S$. It follows that if $F$ is odd then type of $S$ is $k+1$ and if $F$ is even then type of $S$ is $k+2$.
\end{proof}
\begin{corollary}\label{lower bound type alpha}
For a fixed positive integer $\alpha$ and $\epsilon>0$, for sufficiently large $F$
$$T(F,\alpha)>2^{(0.1-\epsilon)F}.$$
\end{corollary}

\section{Some further observations and conjectures}\label{Conjecture}

In this section we describe several observations about counting numerical semigroups of of a given type which are based on numerical evidence.
We start by looking at numerical semigroups of a given Frobenius number and looking at how many of them have odd type, how many have even type.
Table \ref{tab:parity of t given F} has these values for $F\leq 30$, based on this we make the following conjectures.

\begin{conjecture}\label{same parity}
For every $F$
$$\#\{S\mid F(S)=F, t(S)\equiv F(mod\;2)\}>
\#\{S\mid F(S)=F, t(S)\not\equiv F(mod\;2)\}.$$
\end{conjecture}

\begin{conjecture}
$$\lim_{F\to\infty}\frac{\#\{S\mid F(S)=F, t(S)\equiv F(mod\;2)\}}{\#\{S\mid F(S)=F\}} 
=\frac{1}{2}.$$
\end{conjecture}

The data in Table \ref{tab:parity of t given F} has been used to plot the ratio of numerical semigroups in which $t(S)\equiv F(mod\;2)$ among all numerical semigroups with Frobenius number $F$.

\includegraphics[width=0.95\textwidth]{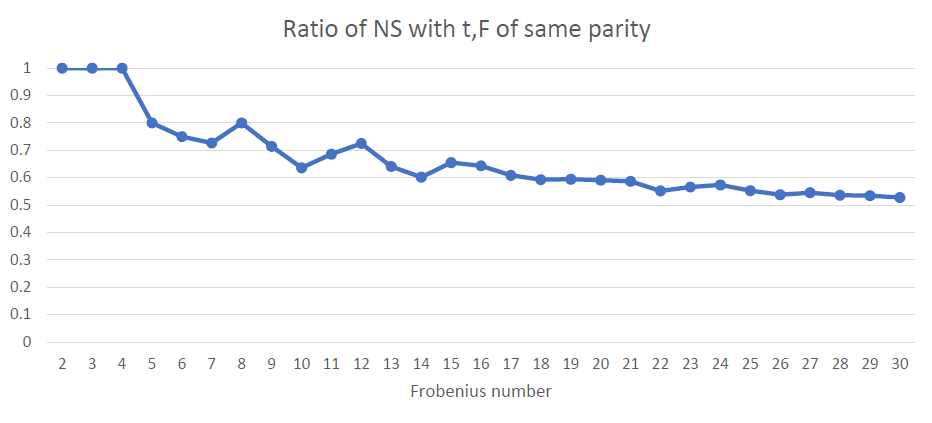}

\begin{table}[h]
    \centering
    \begin{tabular}{|c|c|c|c|c|c|c|c|c|}
    \hline
    F    & Odd t & Even t&F &Odd t &Even t &F &Odd t &Even t\\
    \hline
    2     & 0  & 1  &
        12    &  11 & 29
        &22& 857  &1056\\
    3     & 2  &0  &
        13    &   68& 38
        &23&2320 &1776\\
    4     & 0  & 2 &
        14    &  41 & 62
        &24& 1524  &2054\\
    5     & 4  &1  &
        15    & 131 &69
        &25&4573 &3700\\
    6     &  1 &3 &
        16    &  73 & 132
        &26& 3779  &4396\\
    7     & 8  &3 
        &17& 283  &182
        &27&8803  &7329\\
    8     &  2 &8 
        &18& 165  & 240
        &28&7547 &8720\\
    9     & 15 &6  
        &19& 571  &390
        &29& 18656 & 16247\\
    10    &  8 &  14 
        &20& 368  &532
        &30& 15023 & 16799\\
    11    &  35&16 
        &21& 1073&755
        &&&\\
    \hline
    \end{tabular}
    \caption{Number of numerical semigroups with odd, even type. Counted by Frobenius number.}
    \label{tab:parity of t given F}
\end{table}

Further we conjecture the following about the growth of $T(F,\alpha)$.

\begin{conjecture}
For any $\alpha\geq 2$ and $r\in \{0,1\}$ the following limits exist
$$\lim_{F\equiv r(mod\;2), F\to\infty} \frac{log_{2}(T(F,\alpha))}{F}.$$
\end{conjecture}
When $r=0$ we denote the conjectured limit by $a_{\alpha}$ and when $r=1$ we denote it by $b_{\alpha}$.
Corollary \ref{lower bound type alpha} implies that $a_\alpha, b_\alpha$ are at least $0.1$. Moreover Corollary \ref{AS lower bound} implies that $a_\alpha$ for even $\alpha$ and $b_\alpha$ for odd $\alpha$ are at least $\frac{1}{6}$.

\begin{conjecture}
The two sequences $a_{\alpha}$ and $b_{\alpha}$ are both strictly increasing sequences.
\end{conjecture}

We plot $\frac{log_{2}(T(F,\alpha))}{F}$ for even and odd $F$ and $\alpha=2,3,4,5$.
The limits do appear to exist.

\includegraphics[width=0.85\textwidth]{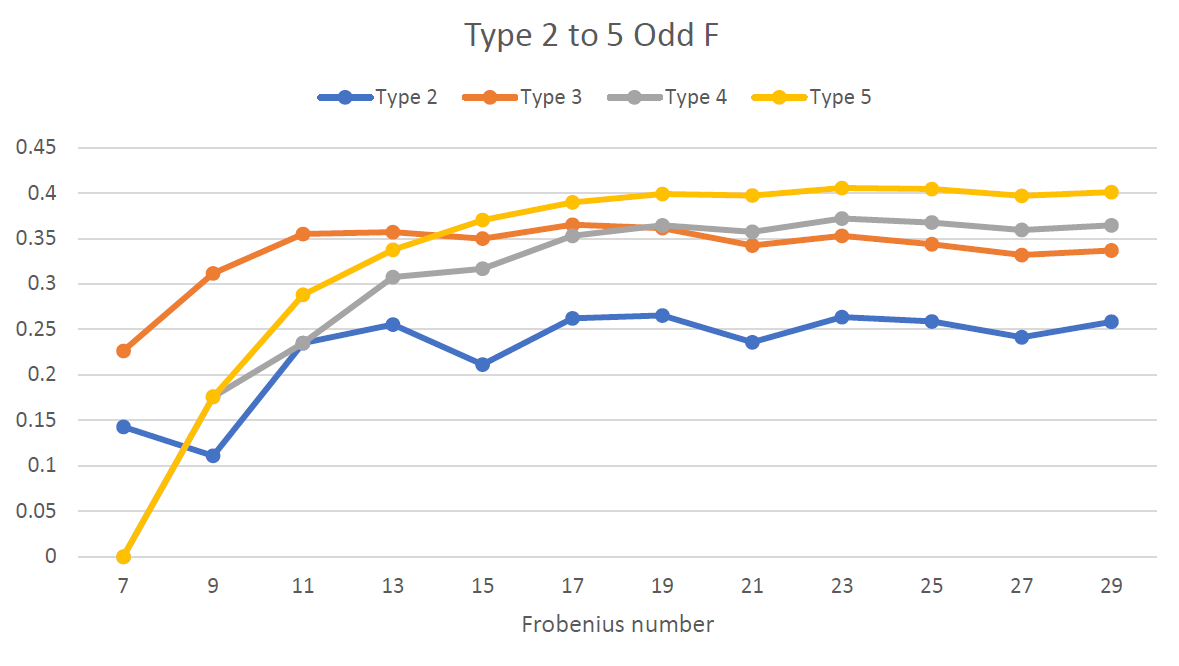}

\includegraphics[width=0.85\textwidth]{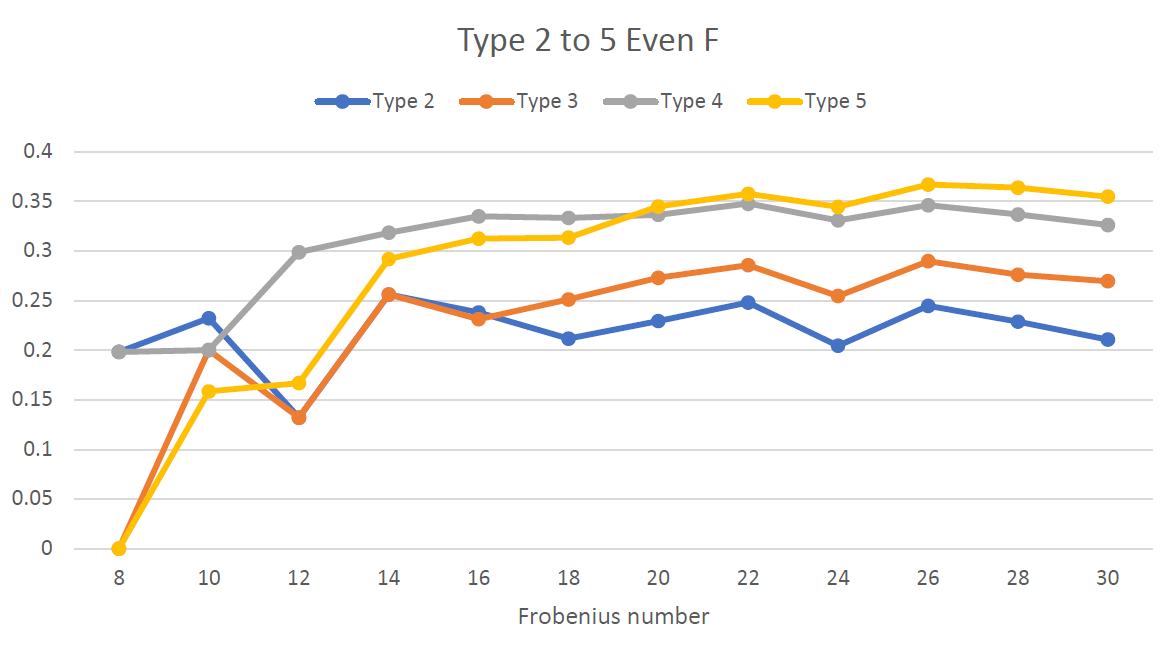}


Similar patterns are observed when we count by genus. However this time numerical semigroups with odd type seem to outnumber the numerical semigroups with even type. Table \ref{tab:parity of t given g} verifies Conjecture \ref{odd t given g} for $g\leq 20$.

\begin{conjecture}\label{odd t given g}
For every $g$ we have
$$\#\{S\mid g(S)=g, t(S)\equiv 1(mod\;2)\}\geq
\#\{S\mid g(S)=g, t(S)\equiv 0 (mod\;2)\}.$$
\end{conjecture}

\begin{conjecture}
$$\lim_{g\to\infty}\frac{\#\{S\mid g(S)=g, t(S)\equiv 1(mod\;2)\}}{\#\{S\mid g(S)=g\}} 
=\frac{1}{2}$$
\end{conjecture}

We plot the ratio of numerical semigroups with odd type among numerical semigroups with genus $g$ based on Table \ref{tab:parity of t given g}.

\includegraphics[width=0.86\textwidth]{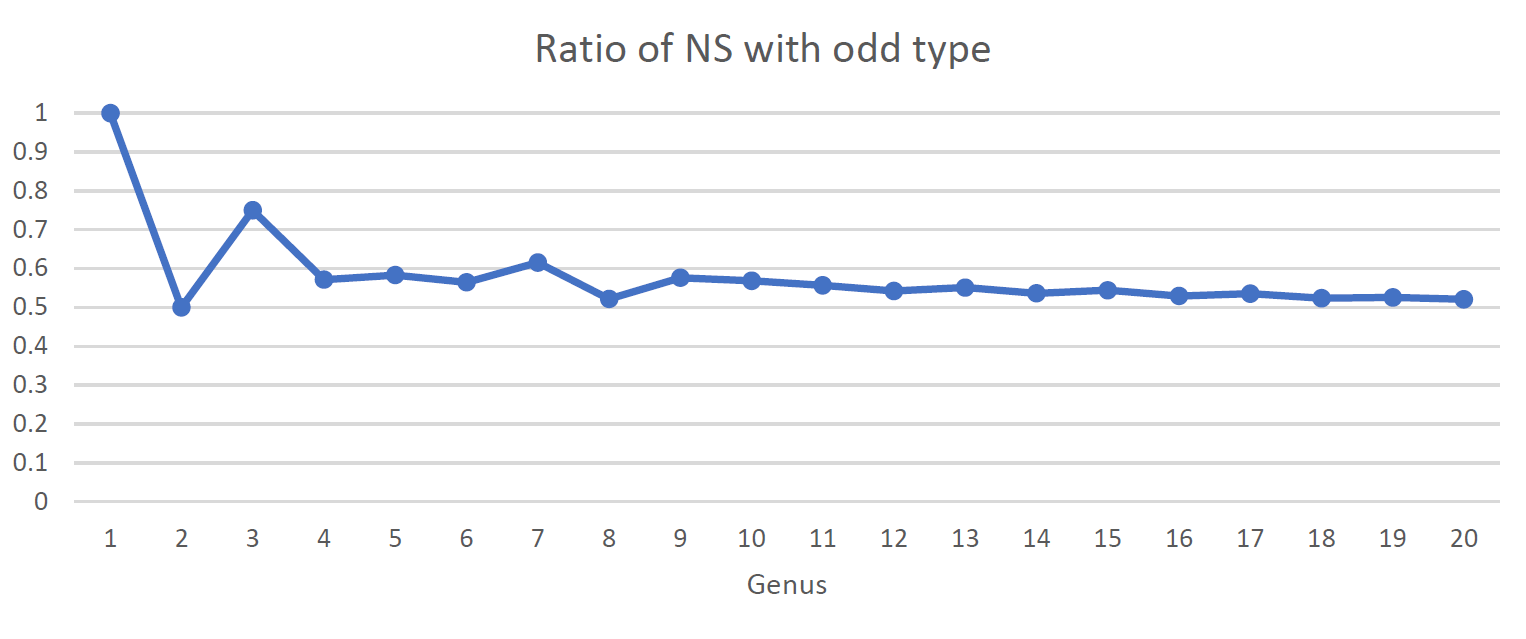}

\begin{table}[h]
    \centering
    \begin{tabular}{|c|c|c|c|c|c|c|c|c|}
         \hline
         g& Odd t & Even t & g& Odd t & Even t & g& Odd t & Even t \\
         \hline
        1 &1 &0 
            &8 &35 &32 
            &15 &1555 &1302 \\
        2 &1 &1 
            &9 &68 &50 
            &16 &2548 &2258 \\
        3 &3 &1 
            &10 &116 &88 
            &17 &4307 &3738 \\
        4 &4 &3 
            &11 &191 &152 
            &18 &7060 &6407 \\
        5 &7 &5 
            &12 &321 &271 
            &19 &11804 &10660 \\
        6 &13 &10 
            &13 &552 &449 
            &20 &19464 &17932 \\
        7 &24 &15 
            &14 &908 &785 
            & & & \\
        \hline
    \end{tabular}
    \caption{Number of numerical semigroups with odd, even type. Counted by genus.}
    \label{tab:parity of t given g}
\end{table}

\begin{conjecture}
For any $\alpha$ the following limits exist
$$\lim_{g\to\infty} \frac{log_{\phi}(L(g,\alpha))}{g}$$
\end{conjecture}
Here $\phi$ is the golden ratio, we denote the limit by $c_\alpha$.

\begin{conjecture}
The sequence $c_\alpha$ is monotonically increasing.
\end{conjecture}

We plot $\frac{log_{\phi}(L(g,\alpha))}{g}$ for $\alpha=1,2,3,4,5$.
The limits do appear to exist.

\includegraphics[width=0.88\textwidth]{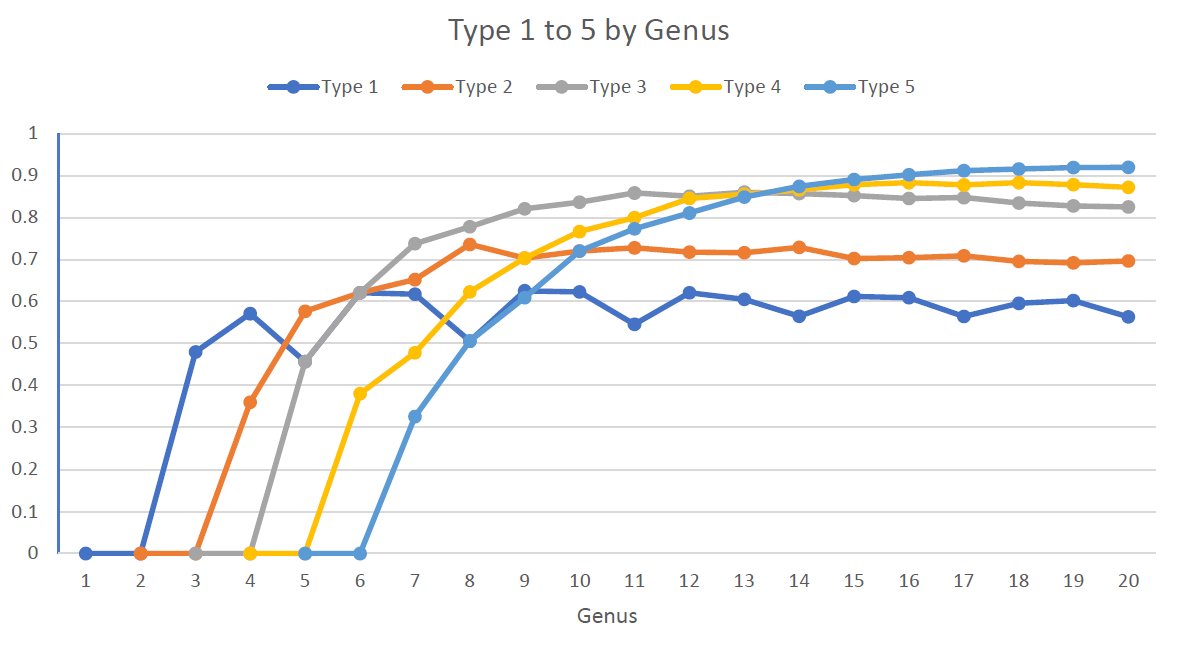}

\end{document}